\documentclass[11pt,a4paper]{amsart}
\usepackage[T1]{fontenc}
\usepackage{lmodern}
\usepackage{amsmath,amssymb,mathtools}
\usepackage[margin=20mm]{geometry}
\usepackage{microtype}
\usepackage[hidelinks]{hyperref}
\allowdisplaybreaks[2]

\newtheorem{theorem}{Theorem}[section]
\newtheorem{proposition}[theorem]{Proposition}
\newtheorem{lemma}[theorem]{Lemma}
\newtheorem{corollary}[theorem]{Corollary}
\theoremstyle{definition}
\newtheorem{definition}[theorem]{Definition}

\theoremstyle{remark}
\newtheorem{remark}[theorem]{Remark}
\numberwithin{equation}{section}

\newcommand{\R}{\mathbb R}
\newcommand{\HH}{\mathbb H}
\newcommand{\SSph}{\mathbb S}
\newcommand{\cH}{\mathcal H}
\newcommand{\cQ}{\mathcal Q}
\newcommand{\dd}{\,\mathrm d}
\newcommand{\e}{\mathrm e}
\newcommand{\Hvec}{\mathbf H}
\DeclareMathOperator{\Vol}{Vol}
\DeclareMathOperator{\Per}{Per}
\DeclareMathOperator{\diam}{diam}
\DeclareMathOperator{\dist}{dist}
\DeclareMathOperator{\spt}{spt}
\DeclareMathOperator{\conv}{conv}
\DeclareMathOperator{\tr}{tr}
\DeclareMathOperator{\diver}{div}
\DeclareMathOperator{\Hess}{Hess}
\DeclareMathOperator{\Lip}{Lip}

\title[Isoperimetry under controlled variation of curvature]
{The Cartan--Hadamard isoperimetric inequality\ under controlled variation of curvature}
\author[G. Wheeler]{Glen Wheeler}
\address{School of Mathematics and Physics, Faculty of Engineering and
Information Sciences, University of Wollongong, Northfields Avenue, Wollongong,
NSW, 2522, Australia}
\email{glenw@uow.edu.au}
\date{13 September 2026}
\hypersetup{
 pdftitle={The Cartan-Hadamard isoperimetric inequality under controlled variation of curvature},
 pdfauthor={Glen Wheeler},
pdfsubject={Sharp Euclidean isoperimetry under global and variable-scale negative sectional curvature pinching},
 pdfkeywords={Cartan-Hadamard, isoperimetric inequality, curvature pinching, curvature variation}
}

\begin{document}
\begin{abstract}
We prove the Cartan-Hadamard conjecture (the sharp Euclidean
isoperimetric inequality) in every dimension, assuming that the sectional curvature is pinched relative to a reference radius of curvature.
The reference radius of curvature may vary, and the condition includes examples
where the radius of curvature blows up, and the manifold has both unbounded
(negative) and asymptotically vanishing sectional curvature.
The key step in the proof is to control the entire convex hull of a quotient
minimiser at its mean-curvature scale and compare the metric distortion
with the surplus in the normalised hyperbolic isoperimetric profile
at that same scale.
\end{abstract}
\maketitle

\section{Introduction}
\label{sec:intro}

Let $(M^n,g)$ be a Cartan--Hadamard manifold; that is, a complete,
simply connected Riemannian manifold with nonpositive sectional
curvature. For a measurable set $E\subset M$ we denote its volume by
$\Vol_g(E)$ and its full ambient perimeter by $\Per_g(E)$. We write
$\omega_n$ for the volume of the unit ball in $\R^n$. The
Cartan--Hadamard isoperimetric conjecture asserts that
\begin{equation}
 \Per_g(E)^n\geq n^n\omega_n\Vol_g(E)^{n-1}
 \label{eq:CH}
\end{equation}
for every bounded set of finite perimeter. The geometric content of
\eqref{eq:CH} is that nonpositive curvature cannot decrease the least
boundary area required to enclose a prescribed volume below its
Euclidean value.

The inequality is known without further curvature assumptions in
dimensions two, three and four; see \cite{Chavel,Croke,Kleiner}.
Chen, Ghomi and Wang \cite{ChenGhomiWang} have now established it
in dimension five for bounded sets of positive volume and finite
perimeter, with equality only for sets isometric to Euclidean balls,
up to sets of measure zero.
Thus, our main theorem (Theorem \ref{thm:main}) is of interest for $n\ge6$.

Kloeckner and Kuperberg \cite{KloecknerKuperberg} give a unified
treatment in dimensions two and four under weaker candle conditions.
Sharp comparisons for small volumes under scalar-curvature and
$C^3$ bounded-geometry hypotheses were obtained by Nardulli and
Osorio Acevedo \cite{NardulliOsorio}. Ghomi and Stavroulakis
\cite{GhomiStavroulakisLocal} prove a local result in every dimension:
a sufficiently small $C^2$ perturbation of the Euclidean metric on
a ball, with nonpositive sectional curvature, cannot decrease that
ball's isoperimetric ratio.

Ghomi and Spruck \cite{GS} established in all dimensions the implication
from a sharp lower bound for the total Gauss--Kronecker curvature of
convex hypersurfaces to \eqref{eq:CH}. In \cite{GhomiPinching} Ghomi proved
that this total-curvature bound holds in every dimension $n\geq4$ for convex
hypersurfaces of sufficiently small diameter
relative to the ambient curvature scale. Our proof uses direct metric
comparison and does not require this total-curvature bound.

Our main theorem is: 

\begin{theorem}
\label{thm:main}
For each integer $n\geq2$ there are constants $\varepsilon_n\in(0,1/2]$ and
$c_n>0$ with the following property.
If $(M^n,g)$ has $(\varepsilon_n,c_n)$-controlled variation of curvature, then
every bounded set $E\subset M$ of finite perimeter and positive volume
satisfies \eqref{eq:CH}.
\end{theorem}

The definition of $(\varepsilon_n,c_n)$-controlled variation of curvature is as follows.

\begin{definition}
\label{def:controlled_curvature_variation}
Let $(M^n,g)$ be a smooth Cartan-Hadamard manifold.
We say that $(M^n,g)$ has \emph{$(\varepsilon_n,c_n)$-controlled variation of curvature} if there exists a positive continuous function $\kappa$ and $\eta\in[0,1)$ with the following properties:
\begin{enumerate}
\item[(i)] $\ell=\kappa^{-1/2}$ is Lipschitz with
\begin{equation}
 \sqrt{1-\eta}-c_n\Lip_g(\kappa^{-1/2})
 \geq\sqrt{1-\varepsilon_n}\text{, and}
 \label{eq:jointcriterion}
\end{equation}
\item[(ii)] For every $x\in M$ and every two-plane $\sigma\subset T_xM$,
\begin{equation}
 -\kappa(x)\leq\sec_g(\sigma)\leq-(1-\eta)\kappa(x).
 \label{eq:pointwisepinching}
\end{equation}
\end{enumerate}
We call the quantity $\ell$ the \emph{reference curvature radius}, and $\kappa$ the \emph{reference curvature}.
\end{definition}

The Lipschitz constant of $\ell$ is invariant under constant metric rescaling.
For $C^1$ functions, $\Lip_g\ell\leq\tau$ is equivalent to
$|\nabla\kappa|\leq2\tau\kappa^{3/2}$. 

We now make a number of remarks.

\begin{remark}
\label{rem:varyingscale}
Taking $M=\R\times\R^{n-1}$, $n\geq3$, with the warped product metric
\[
 g=\dd t^2+f(t)^2g_{\R^{n-1}},\qquad
 f(t)=\exp\left(\int_0^t\frac{\dd s}{L(s)}\right),\qquad
 L(t)=\sqrt{1+\delta^2t^2}-\delta t,
\]
gives examples with asymptotically vanishing and unbounded negative
sectional curvature on the same manifold. Here $\delta>0$ is a
small parameter. The function $L$ is positive and smooth, with
$-2\delta<L'<0$, and differentiating $f$ gives
$f'/f=L^{-1}$ and $f''/f=(1-L')L^{-2}$.
The sectional curvatures of planes tangent to the Euclidean factor
and planes containing $\partial_t$ are, respectively,
\[
 K_{\mathrm{tan}}=-L^{-2},\qquad
 K_{\mathrm{rad}}=-(1-L')L^{-2}.
\]
The curvature operator is diagonal on radial and tangential
two-forms, so the sectional curvature of every two-plane is a convex
combination of these values.
Thus, taking
\[
 \kappa=(1+2\delta)L^{-2},\qquad
 \eta=\frac{2\delta}{1+2\delta},
\]
we have
\[
 -\kappa\leq\sec_g\leq-\frac{\kappa}{1+2\delta}<0,\qquad
 |\nabla(\kappa^{-1/2})|
 =\frac{|L'|}{\sqrt{1+2\delta}}
 \leq\frac{2\delta}{\sqrt{1+2\delta}}.
\]
Thus \eqref{eq:pointwisepinching} holds and
$\Lip_g(\kappa^{-1/2})\leq2\delta/\sqrt{1+2\delta}$.
In particular, \eqref{eq:jointcriterion} holds whenever
$0<\delta\leq\varepsilon_n/[2(1+2c_n)]$, since
\[
 \sqrt{1-\eta}-c_n\Lip_g(\kappa^{-1/2})
 \geq\frac{1-2c_n\delta}{\sqrt{1+2\delta}}
 \geq1-(1+2c_n)\delta
 \geq1-\frac{\varepsilon_n}{2}
 \geq\sqrt{1-\varepsilon_n}.
\]
The metric is complete: escape in the $t$ direction has infinite
length, while on each bounded $t$ interval the fibre metric has
a positive Euclidean lower bound. It is also simply connected,
and hence is Cartan--Hadamard. As $t\to-\infty$ we have
$L(t)\to\infty$, so all sectional curvatures tend to zero.
As $t\to+\infty$ we have $L(t)\to0$, so all sectional curvatures
tend to $-\infty$. In particular, no constants $0<a\leq b<\infty$
give global bounds $-b^2\leq\sec_g\leq-a^2$. This shows that
Theorem~\ref{thm:main} applies beyond global negative pinching.
For $n\geq6$ these examples illustrate the additional isoperimetric
range of the theorem; in dimensions three through five the
isoperimetric conclusion already holds without curvature pinching.
\end{remark}

\begin{remark}
Taking $\kappa\equiv b^2$ and $\eta=\varepsilon_n$ gives the
global pinching case
\begin{equation}
 -b^2\leq\sec_g\leq-(1-\varepsilon_n)b^2.
 \label{eq:pinching}
\end{equation}
The curvature radius is then constant, so
\eqref{eq:jointcriterion} is automatic and no bounds on derivatives
of curvature are required.
\end{remark}

\begin{remark}
\label{rem:constants}
The dimension-dependent constants $\varepsilon_n$ and $c_n$ may be
chosen as follows. Put $s_{n-1}=n\omega_n$, and set
\[
 d_n=\frac{6\e (n-1)^{n-1}s_{n-1}}{\omega_{n-1}}.
\]
With $T_n=\sqrt2\,n/(n-1)$, an admissible choice is
\[
 \begin{aligned}
 c_n&=\frac{n d_n}{(n-1)^2}
     =\frac{6\e n^2(n-1)^{n-3}\omega_n}{\omega_{n-1}},\\
 \varepsilon_n&=\min\left\{\frac12,
 \frac{3(n-1)^2}{(n+2)\cosh T_n(1+\cosh T_n)d_n^2}\right\}.
 \end{aligned}
\]
These expressions give $c_n>0$ and $0<\varepsilon_n\leq1/2$.
All quantities are specified by the dimension. We do not optimise
these choices.
\end{remark}

\begin{remark}
The proof of Theorem~\ref{thm:main} uses the bounds for curvature
and the variation of $\kappa^{-1/2}$ only on the selected hull $C$.
In particular, if $\kappa$ and the same hypotheses are given on a
closed geodesic ball $\overline B$, then every finite-perimeter
set in $B$ satisfies \eqref{eq:CH}. One minimises in $B$ and
argues as in Corollary~\ref{cor:localball}.
\end{remark}

We now give some words on the proof.
The key new idea in our proof is to exploit control of the entire convex hull
of a minimiser of the isoperimetric quotient at the mean-curvature scale
in a quantitative hyperbolic comparison.
We use the constrained minimisation and convex-hull framework of Ghomi--Spruck,
minimising the quotient over subsets of a compact carrier and allowing the
volume to vary. The resulting relation between perimeter, volume and mean
curvature makes it possible to control the entire candidate before choosing a
comparison scale.
Write
\begin{equation}
 \cQ_g(E)=\frac{\Per_g(E)^n}{n^n\omega_n\Vol_g(E)^{n-1}}.
 \label{eq:quotient}
\end{equation}
A set with quotient less than one gives a quotient minimiser
$\Omega$ among subsets of a sufficiently large geodesic ball. Writing
$P=\Per_g(\Omega)$, $V=\Vol_g(\Omega)$ and $H=(n-1)P/(nV)$, we show
that its boundary has connected support and generalised mean-curvature
vector bounded by $H$; see Proposition~\ref{prop:minimiser}.
Writing $\theta=\cQ_g(\Omega)<1$, monotonicity then gives
\[
 \diam C\leq\frac{d_n\theta}{H},\qquad
 \Vol_g(\Omega)=\omega_n\theta\left(\frac{n-1}{H}\right)^n,
 \qquad C=\conv(\partial\Omega).
\]
After rescaling to $-1\leq\sec_g\leq-\lambda$, with
$\lambda\geq1/2$, Rauch comparison controls the metric on the hull
in hyperbolic exponential coordinates. Its logarithmic distortion is
bounded by a constant times $(1-\lambda)(\diam C)^2$.
The logarithm of the Euclidean-normalised hyperbolic isoperimetric
profile is bounded below by a positive dimensional constant times
$\Vol_g(\Omega)^{2/n}$ on the relevant volume range. The two
displayed scale estimates allow the latter quantity to absorb the
distortion when the pinching is sufficiently tight. This comparison
holds at every scale, including as the volume tends to zero.

This differs from the argument of Chen--Ghomi--Wang
\cite{ChenGhomiWang}, whose chord-integral identities and Jacobi-field
estimates yield a sharp area bound for constant-mean-curvature
hypersurfaces in dimension five. Their extension to hypersurfaces
confined by a geodesic ball directly excludes the quotient minimiser
in that dimension; we give the short comparison after
Proposition~\ref{prop:minimiser}. Our quantitative comparison at the
mean-curvature scale applies in every dimension and proves the
isoperimetric inequality without a total-curvature estimate for
convex hypersurfaces.

The localisation is also what permits a varying curvature scale.
All the strict curvature bounds in the contradiction argument are
needed only on the convex hull $C$. If
$k_-=\min_C\kappa$ and $k_+=\max_C\kappa$, radial comparison gives
$H\geq(n-1)^2\sqrt{(1-\eta)k_-}/n$. Together with the diameter bound,
this converts a small Lipschitz constant for $\kappa^{-1/2}$ into pinching
throughout $C$.

The article is organised as follows.
In Section~\ref{sec:prelim} we fix notation and conventions.
Section~\ref{sec:minimiser} records the required regularity and
establishes the constrained minimisation and scale estimates.
Section~\ref{sec:comparison}
contains the hyperbolic comparison and the exclusion argument on the
convex hull. In Section~\ref{sec:proof} we complete the proof of
Theorem~\ref{thm:main}.

\section{Notation and preliminary estimates}
\label{sec:prelim}

Throughout the paper $s_{n-1}=|\SSph^{n-1}|=n\omega_n$, and all
distances and diameters are ambient. A convex body is a compact
geodesically convex set with nonempty interior. We use $\conv(X)$ to
denote the closed convex hull of $X$. Closed metric balls are convex
and compact in a Cartan--Hadamard manifold, so the hull of a bounded
set is compact.

For a set $E$ of finite perimeter, $\partial^*E$ denotes its reduced
boundary and $\nu_E$ its measure-theoretic outward unit normal. Thus
\[
 \Per_g(E)=\cH_g^{n-1}(\partial^*E).
\]
The perimeter here includes any portion supported on the boundary of
a constraining ball. We use the compactness, lower semicontinuity,
Gauss--Green and diffeomorphism variation theorems for finite-perimeter
sets in their local Riemannian form; these follow in smooth charts
from the corresponding Euclidean results in \cite{Maggi}.

On an oriented regular hypersurface $\Gamma$ we take the shape operator to be
$A(X)=\nabla_X\nu$ and the scalar mean curvature to be
$H_\Gamma=\tr A$. These conventions give positive principal
curvatures on outward-oriented Euclidean spheres. If $W$ is the
multiplicity-one varifold associated to $\Gamma$ with weight measure $\mu$, its
first variation is
\[
 \delta W(X)=\int\diver_{T_xW}X\dd\mu(x).
\]
Our convention for its generalised mean-curvature vector is
$\delta W(X)=-\int\langle\Hvec,X\rangle\dd\mu$ when the first
variation is absolutely continuous. On a regular hypersurface,
$\Hvec=-H_\Gamma\nu$.

\section{Constrained minimisers and scale estimates}
\label{sec:minimiser}

We first record the regularity of an isoperimetric region in an
open geodesic ball $B$. By \cite[Lemma 7.2]{GS}, such a region has an open
representative $\Omega$ for which $\Gamma=\partial\Omega$ is smooth in $B$
outside a closed singular set $S$ of Hausdorff dimension at most $n-8$. The
free regular boundary has constant scalar mean curvature $H_0$. Near $\partial
B$, the boundary is $C^{1,1}$ and $H_\Gamma\leq H_0$ almost everywhere.
In particular $S$ is empty when $n\leq7$. The standard density estimates (see
\cite[Theorem 21.11]{Maggi} for the Euclidean statement), together with the
preceding regularity and the dimension bound on $S$, give
\begin{equation}
 \Gamma=\spt(\cH_g^{n-1}\llcorner\partial^*\Omega),
 \qquad
 M\setminus\Omega=\overline{M\setminus\overline\Omega},
 \qquad \cH_g^{n-1}(S)=0.
 \label{eq:representative}
\end{equation}
We now associate to any violation of \eqref{eq:CH} a constrained
minimiser with controlled first variation and connected boundary.
The connectedness assertion in the next proposition will be used in
the covering argument of Lemma~\ref{lem:packing}.

\begin{proposition}
\label{prop:minimiser}
Suppose that $\cQ_g(E)<1$ for some bounded finite-perimeter set of
positive volume in a Cartan--Hadamard manifold $M^n$. There is an
open geodesic ball $B$ and a quotient minimiser $\Omega\subset B$,
with
\[
 P=\Per_g(\Omega),\qquad V=\Vol_g(\Omega),\qquad
 \theta=\cQ_g(\Omega)<1,\qquad H=\frac{(n-1)P}{nV}>0,
\]
whose perimeter varifold $W$ has compact connected support
$\Gamma=\partial\Omega$. Its first variation satisfies
\begin{equation}
 \delta W=-\Hvec\mu,\qquad |\Hvec|\leq H\quad\mu\text{-almost everywhere}.
 \label{eq:meanbound}
\end{equation}
The convex body $C=\conv(\Gamma)$ contains $\Omega$. Furthermore,
\begin{equation}
 P=s_{n-1}\theta\left(\frac{n-1}{H}\right)^{n-1},\qquad
 V=\omega_n\theta\left(\frac{n-1}{H}\right)^n.
 \label{eq:PV}
\end{equation}
\end{proposition}

\begin{proof}
Choose $B=B_R(o)$ containing $E$ compactly and minimise $\cQ_g$
over its positive-volume finite-perimeter subsets. We first exclude
loss of volume at either endpoint.

On a fixed compact smooth neighbourhood of $\overline B$, the local
Euclidean $BV$ inequality and a partition of unity give, for every
$\eta>0$, a constant $c_\eta$ such that
\begin{equation}
 n\omega_n^{1/n}\Vol_g(F)^{(n-1)/n}
 \leq(1+\eta)\Per_g(F)+c_\eta\Vol_g(F)
 \label{eq:localBV}
\end{equation}
for all $F\subset B$ of finite perimeter.
To see this: Choose finitely many charts with metric distortion sufficiently
small in terms of $\eta$, and a nonnegative subordinate partition
$\{\phi_j\}$ whose sum is one near $\overline B$. Apply the
Euclidean inequality to each $\phi_j\chi_F$ and sum. The triangle
inequality in $L^{n/(n-1)}$ bounds the left-hand side, the terms
$\phi_j|D\chi_F|$ sum to the perimeter, and the derivatives of
$\phi_j$ contribute at most $c_\eta\Vol_g(F)$. This proves
\eqref{eq:localBV}. Dividing by $\Vol_g(F)^{(n-1)/n}$, letting the
volume tend to zero and then $\eta\downarrow0$, we obtain
\[
 \liminf_{\Vol_g(F)\to0}\cQ_g(F)\geq1.
\]

At the other endpoint, $\chi_F\to\chi_B$ in $L^1$ as
$\Vol_g(F)\to\Vol_g(B)$, so lower semicontinuity reduces the
question to $\cQ_g(B)\geq1$. We verify this directly from the Jacobi
equation. In geodesic polar coordinates about $o$, write
\[
 \dd\Vol_g=\mathcal J(r,\xi)\dd r\dd\sigma(\xi),\qquad
 A(r)=\cH_g^{n-1}(\partial B_r(o))
     =\int_{\SSph^{n-1}}\mathcal J(r,\xi)\dd\sigma(\xi),
\]
where $\xi$ is a unit vector in $T_oM$ and $\dd\sigma$ is the
standard spherical measure. Thus $\mathcal J$ is the radial volume
density; its Euclidean value is $r^{n-1}$.

Fix $\xi$ and let $\gamma(r)=\exp_o(r\xi)$. For $v\perp\xi$, let
$Y_v$ be the perpendicular Jacobi field with $Y_v(0)=0$ and
$D_rY_v(0)=v$, where $D_r$ denotes covariant differentiation along
$\gamma$. For $v\ne0$ the field has no further zero, since $M$ is
Cartan--Hadamard. The Jacobi equation and Cauchy--Schwarz give, with
$f=|Y_v|$,
\[
 ff''=|D_rY_v|^2-(f')^2
       -\sec_g(\dot\gamma\wedge Y_v)f^2\geq0.
\]
Thus $f$ is convex and $f(0)=0$, so $f(r)/r$ is nondecreasing.
Choose an orthonormal basis $e_1,\ldots,e_{n-1}$ of $\xi^\perp$.
Applying this to every $v\in\xi^\perp$ shows that the positive
definite matrix
\[
 G(r)=\bigl(r^{-2}\langle Y_{e_i}(r),Y_{e_j}(r)\rangle\bigr)_{i,j=1}^{n-1}
\]
is nondecreasing as a quadratic form. Its determinant is therefore
nondecreasing, and $G(r)\to I$ as $r\downarrow0$. Consequently
\[
 \frac{\mathcal J(r,\xi)}{r^{n-1}}=\sqrt{\det G(r)}
\]
is nondecreasing in $r$, with limit one at the origin.
Integrating over $\xi$ gives
$A(r)\leq(r/R)^{n-1}A(R)$ for $0<r\leq R$ and
$A(R)\geq s_{n-1}R^{n-1}$. Hence
\[
 \Vol_g(B_R)=\int_0^R A(r)\dd r\leq\frac{R}{n}A(R),
\]
and therefore
\[
 \cQ_g(B_R)\geq\frac{A(R)}{s_{n-1}R^{n-1}}\geq1.
\]
Since the infimum is less than one, this implies that a minimising sequence
must have volume bounded away from zero and $\Vol_g(B)$. Its perimeters are
bounded, so $BV$ compactness and lower semicontinuity give a minimiser $\Omega$
with $0<V<\Vol_g(B)$ and quotient $\theta<1$.  It also minimises perimeter at
volume $V$. We may assume that $\Omega$ is open: By the regularity and density
results recalled above, modifying $\Omega$ on a set of zero $g$-volume so
that it is open and $\Gamma=\partial\Omega$ has the stated regularity and
satisfies \eqref{eq:representative} is valid, and leaves its volume, perimeter
and quotient unchanged.

Every smooth vector field compactly supported in $B$ generates
two-sided admissible variations. Differentiating \eqref{eq:quotient}
under its flow yields
\begin{equation}
 \delta W(X)=H\int_{\partial^*\Omega}
                   \langle X,\nu_\Omega\rangle\dd\mu.
 \label{eq:stationary}
\end{equation}
This is an identity for the perimeter varifold on all of $B$,
including its singular points. On the free regular part it gives
$H_\Gamma=H$. That part is nonempty, since a proper positive-volume
region in the connected ball has an interior interface of positive
perimeter. Thus the constant $H_0$ in the regularity theorem is $H$.

We now conduct a mean curvature comparison argument.
Near $\partial B$ we have $C^{1,1}$ regularity and
$H_\Gamma\leq H$ almost everywhere. At obstacle contact write the
two hypersurfaces as $C^{1,1}$ graphs $u$ and $v$. The function
$u-v$ has one sign and vanishes on the contact set, so $Du=Dv$
there. A Lipschitz function has derivative zero almost everywhere on
each of its level sets; applied to $Du-Dv$, this also gives
$D^2u=D^2v$ almost everywhere on contact. The two outward normals
agree. Therefore
\[
 H_\Gamma=H_{\partial B}\geq0
 \quad\text{almost everywhere on }\Gamma\cap\partial B.
\]
Together with $H_\Gamma=H>0$ on the free part, this proves
$0\leq H_\Gamma\leq H$ almost everywhere. A $C^{1,1}$
hypersurface has absolutely continuous first variation.
Patching its first-variation formula near $\partial B$ with
\eqref{eq:stationary} proves \eqref{eq:meanbound} globally.

We next prove connectedness. Each open component $\Omega_i$ of
$\Omega$ has boundary contained in $\Gamma$, which has finite
$\cH_g^{n-1}$ measure, and consequently has finite perimeter. At every
regular boundary point the connected interior side of a separating
graph patch belongs to exactly one component. The exceptional set
has $\cH_g^{n-1}$ measure zero. Thus, modulo $\cH_g^{n-1}$-null sets,
\[
 \partial^*\Omega=\mathop{\dot\bigcup}_i\partial^*\Omega_i,
 \qquad \nu_{\Omega_i}=\nu_\Omega
 \quad\text{on }\partial^*\Omega_i.
\]
There are at most countably many open components, so perimeter and
volume are additive:
\[
 P=\sum_iP_i,\qquad V=\sum_iV_i.
\]
Quotient minimality gives
\[
 P_i\geq n\omega_n^{1/n}\theta^{1/n}V_i^{(n-1)/n}.
\]
If there were two nonempty components, the strict inequality
$\sum_iV_i^{(n-1)/n}>(\sum_iV_i)^{(n-1)/n}$ would contradict the
definition of $\theta$. Thus $\Omega$ is connected.

Suppose that $D$ is a bounded component of the open exterior
$M\setminus\overline\Omega$. It lies in $B$, since the complement
of the closed ball is connected and belongs to the unbounded
exterior component. As above, $D$ has finite perimeter. At each
regular interface its outward normal is opposite to that of
$\Omega$. Since the singular set contributes no perimeter, we have
$\partial^*D\subset\partial^*\Omega$ modulo $\cH_g^{n-1}$-null sets
and $\nu_D=-\nu_\Omega$ there. The corresponding terms cancel in
$D\chi_{\Omega\cup D}=D\chi_\Omega+D\chi_D$. Hence
\[
 \Per_g(\Omega\cup D)=P-\Per_g(D),\qquad
 \Vol_g(\Omega\cup D)=V+\Vol_g(D)>V.
\]
This strictly decreases the quotient, a contradiction. There are no
bounded exterior components, and all unbounded ones meet the
connected complement of a sufficiently large ball. The open exterior
is therefore connected. By \eqref{eq:representative},
$M\setminus\Omega$ is connected as well. The theorem of
Czarnecki--Kulczycki--Lubawski \cite{Connected} states that an
open connected subset of $\R^n$ has connected boundary if and
only if its complement is connected. Both hypotheses hold for
$\Omega$ under the Cartan--Hadamard diffeomorphism $M\simeq\R^n$.
Thus $\Gamma$ is connected.

To see that $\Omega\subset C$, take a complete geodesic through
any point of $\Omega$. Its first exits in the two directions lie on
$\Gamma$, and the intervening segment contains the chosen point.
Thus $C$ has nonempty interior.
Finally, substituting $P=nHV/(n-1)$ into \eqref{eq:quotient} gives
\eqref{eq:PV}.
\end{proof}

\begin{remark}
In dimension five, the confined constant-mean-curvature inequality
of Chen--Ghomi--Wang \cite[Proposition~5.1]{ChenGhomiWang}
already rules out this minimiser. There are no singular points in
this dimension, so $\Gamma$ is $C^{1,1}$, with smooth free part.
Their scalar mean curvature convention agrees with ours: the free
part has mean curvature $H>0$, and the mean curvature on obstacle
contact is at most $H$ almost everywhere. Their proposition gives
\[
 H^4P\geq4^4s_4,
\]
contradicting $H^4P=4^4s_4\theta<4^4s_4$. 
\end{remark}

We next estimate the diameter in terms of the mean curvature.
The argument uses ambient balls and therefore does not
require an intrinsic distance on a singular hypersurface.

\begin{lemma}
\label{lem:packing}
Let $W$ be a nonzero integral $k$-varifold, $k\geq1$, in a
Cartan--Hadamard manifold. Suppose that its support is compact and
connected and that
\[
 \delta W=-\Hvec\mu,\qquad |\Hvec|\leq L
 \quad\mu\text{-almost everywhere},
\]
where $L>0$. Then
\begin{equation}
 \mu(M)\geq\frac{\omega_k}{\e L^k},\qquad
 \diam(\spt\mu)\leq\frac{6\e}{\omega_k}L^{k-1}\mu(M).
 \label{eq:packing}
\end{equation}
\end{lemma}

\begin{proof}
We first recall the monotonicity calculation in the form required
here. Fix $x\in M$, set $r=\dist(x,\cdot)$ and
$v(t)=\mu(B_t(x))$. Hessian comparison gives
\[
 \tr_T\Hess\left(\frac{r^2}{2}\right)\geq k
\]
on every $k$-plane $T$. Choose a smooth, compactly supported,
nonnegative and nonincreasing radial cutoff $\phi$, constant near
zero. First variation with $X=\phi(r)r\nabla r$ gives
\[
 k\int\phi(r)\dd\mu
 +\int r\phi'(r)|\nabla^W r|^2\dd\mu
 \leq L\int r\phi(r)\dd\mu.
\]
Since $\phi'\leq0$ and $|\nabla^W r|\leq1$, this implies
\[
 k\int\phi(r)\dd\mu
 \leq-\int r\phi'(r)\dd\mu+L\int r\phi(r)\dd\mu.
\]
Approximating the indicator of $[0,t)$, we obtain for almost every $t>0$
\[
 kv(t)\leq t v'_{\mathrm{ac}}(t)+Lt v(t).
\]
The nondecreasing radial mass $v$ is locally $BV$, and its singular
derivative is nonnegative. It follows, in the sense of measures on
$(0,\infty)$, that
\[
 D\bigl(\e^{Lt}t^{-k}v(t)\bigr)\geq0.
\]
At $\mu$-almost every centre, integrality gives density at least
one. Hence at each such centre
\[
 \mu(B_t(x))\geq\omega_k\e^{-Lt}t^k\quad(t>0).
\]
The same bound holds at every point of the support. Indeed, for
$x\in\spt\mu$ choose density points $x_j\to x$ and use
\[
 B_{t-\dist(x_j,x)}(x_j)\subset B_t(x)
\]
for all sufficiently large $j$, then pass to the limit.
In particular, at $\rho=L^{-1}$,
\begin{equation}
 \mu(B_\rho(x))\geq\omega_k\e^{-1}\rho^k
 \quad\text{for every }x\in\spt\mu.
 \label{eq:ballmass}
\end{equation}
This proves the first inequality in \eqref{eq:packing}.

Choose a maximal disjoint family
$\{B_\rho(x_i)\}_{i=1}^N$ with centres in $\spt\mu$.
The lower bound \eqref{eq:ballmass} ensures that the selection
terminates and that
\[
 N\omega_k\e^{-1}\rho^k\leq\mu(M).
\]
Maximality implies that the balls $B_{3\rho}(x_i)$ cover the
support. Their intersections with the support form a finite relatively
open cover. Its intersection graph is connected, since the support is
connected. A simple chain joining balls containing any two chosen
support points has at most $N$ vertices. Consecutive centres are
less than $6\rho$ apart, while the two endpoint contributions are
less than $3\rho$ each. Consequently
\[
 \diam(\spt\mu)\leq6N\rho
 \leq\frac{6\e}{\omega_k}\rho^{1-k}\mu(M),
\]
as required.
\end{proof}

\begin{corollary}
\label{cor:scale}
For the minimiser in Proposition~\ref{prop:minimiser}, set
\begin{equation}
 d_n=\frac{6\e (n-1)^{n-1}s_{n-1}}{\omega_{n-1}}.
 \label{eq:diameterconstant}
\end{equation}
Then
\begin{equation}
 \diam C=\diam\Gamma\leq\frac{d_n\theta}{H}.
 \label{eq:scales}
\end{equation}
If in addition $\sec_g\leq-\Lambda<0$ at every point of $C$, then
\begin{equation}
 H\geq\frac{(n-1)^2}{n}\sqrt\Lambda.
 \label{eq:lowerH}
\end{equation}
\end{corollary}

\begin{proof}
Apply Lemma~\ref{lem:packing} with $k=n-1$ and $L=H$. Its
diameter estimate and \eqref{eq:PV} give
\[
 \diam\Gamma\leq\frac{6\e}{\omega_{n-1}}H^{n-2}P
 =d_n\theta H^{-1}\leq d_nH^{-1}.
\]
To check the hull diameter, let $d=\diam\Gamma$. For every
$x\in\Gamma$ the closed convex ball $\overline B_d(x)$ contains
$\Gamma$ and hence $C$. It follows that, for any $y\in C$,
$\Gamma\subset\overline B_d(y)$, so again
$C\subset\overline B_d(y)$. Thus $\diam C\leq d$, and the
reverse inequality follows from $\Gamma\subset C$.

Finally choose $p\in\Omega$ and put $r=\dist(p,\cdot)$. Every
segment from $p$ to a point of $C$ lies in $C$, so Hessian comparison
gives $\Delta r\geq (n-1)\sqrt\Lambda$ on $C\setminus\{p\}$.
For all sufficiently small $t>0$ we have
$\overline B_t(p)\subset\Omega$. Gauss--Green on
$\Omega\setminus\overline B_t(p)$, using the inward radial normal
on its inner sphere, gives
\[
 (n-1)\sqrt\Lambda\bigl(V-\Vol_g(B_t(p))\bigr)
 \leq P-\cH_g^{n-1}(\partial B_t(p))\leq P=\frac{nHV}{n-1}.
\]
Letting $t\downarrow0$ proves \eqref{eq:lowerH}.
\end{proof}

\section{Hyperbolic comparison and exclusion on the convex hull}
\label{sec:comparison}

Let $h$ be the metric of constant sectional curvature $-1$ on
$\HH^n$. Its isoperimetric profile $I_h$ is realised by geodesic
balls \cite[Chapter VI]{Chavel}. We write
\[
 V_h(r)=s_{n-1}\int_0^r\sinh^{n-1} t\dd t,\qquad
 A_h(r)=s_{n-1}\sinh^{n-1} r,
\]
so that $I_h(V_h(r))=A_h(r)$. Define the Euclidean-normalised
profile quotient by
\begin{equation}
 Q_h(v)=\frac{I_h(v)^n}{n^n\omega_n v^{n-1}},\qquad v>0.
 \label{eq:Qh}
\end{equation}
The quantitative bound in the next lemma will allow us to absorb
the metric distortion at the scale of the minimiser.

\begin{lemma}
\label{lem:margin}
The function $Q_h$ is strictly increasing on $(0,\infty)$, with
$\lim_{v\downarrow0}Q_h(v)=1$. In particular $Q_h(v)>1$ for
every $v>0$. Moreover, if $v=\omega_na^n$ and $0<a\leq T$, then
\begin{equation}
 \log Q_h(v)\geq
 \frac{n(n-1)a^2}{(n+2)\cosh T(1+\cosh T)}.
 \label{eq:quantitativemargin}
\end{equation}
\end{lemma}

\begin{proof}
The limit follows from $\sinh r/r\to1$ at zero. Put $A=A_h(r)$
and $V=V_h(r)$. Since $A'=(n-1)\coth r\,A$ and $V'=A$, we have
\[
 \frac{\dd}{\dd r}\log Q_h(V_h(r))
 =(n-1)\left(n\coth r-\frac{A}{V}\right).
\]
On the other hand,
\[
 \frac{\dd}{\dd r}\left(V-\frac{A\tanh r}{n}\right)
 =\frac{A\tanh^2r}{n}>0.
\]
The expression in parentheses vanishes at zero, and hence
$nV\coth r>A$ for $r>0$. This proves strict monotonicity.

To prove \eqref{eq:quantitativemargin}, choose $r>0$ with
$V_h(r)=v$ and put $\rho=\sinh r$. Since $V_h(r)\geq\omega_nr^n$,
we have $r\leq a\leq T$. Changing variables gives
\[
 a^n=n\int_0^\rho\frac{t^{n-1}}{\sqrt{1+t^2}}\dd t,
\]
and hence
\begin{align*}
 \rho^n-a^n
 &=n\int_0^\rho
 \frac{t^{n+1}}{\sqrt{1+t^2}(1+\sqrt{1+t^2})}\dd t\\
 &\geq\frac{n\rho^{n+2}}{(n+2)\cosh T(1+\cosh T)}.
\end{align*}
Now $Q_h(v)=(\rho/a)^{n(n-1)}$ and $\rho\geq a$.
Applying $-\log(1-z)\geq z$ to $z=1-a^n/\rho^n$ yields
\[
 \log Q_h(v)\geq(n-1)\frac{\rho^n-a^n}{\rho^n}
 \geq\frac{n(n-1)a^2}{(n+2)\cosh T(1+\cosh T)},
\]
as required.
\end{proof}

\begin{lemma}
\label{lem:chart}
Let $C$ be a convex body in a smooth Cartan--Hadamard manifold.
Suppose that $-1\leq\sec_g\leq-\lambda<0$ at every point of $C$,
where $0<\lambda\leq1$. Let $F\subset C$ be a positive-volume
set of finite perimeter. Choose $p\in C$, $R\geq\diam C$, and set
\begin{equation}
 \alpha_\lambda(R)
 =\frac{\sinh(\sqrt\lambda R)}{\sqrt\lambda\sinh R}.
 \label{eq:alpha}
\end{equation}
Then
\begin{equation}
 \cQ_g(F)\geq\alpha_\lambda(R)^{n(n-1)}Q_h(\Vol_g(F)).
 \label{eq:chart}
\end{equation}
Moreover, for $0<\lambda\leq1$,
\begin{equation}
 \log\alpha_\lambda(R)
 \geq-\frac{(1-\lambda)R^2}{6}.
 \label{eq:alphabound}
\end{equation}
\end{lemma}

\begin{proof}
Identify $M$ and $\HH^n$ by exponential coordinates at $p$ and
at a hyperbolic origin, using a linear isometry of their tangent
spaces. For $x\in C$, the radial segment $[p,x]$ stays in $C$.
The Gauss lemma and Rauch comparison along this segment give, at $x$
in polar coordinates,
\begin{equation}
 \dd r^2+\frac{\sinh^2(\sqrt\lambda r)}{\lambda}
                 g_{\SSph^{n-1}}
 \leq g\leq
 h=\dd r^2+\sinh^2r\,g_{\SSph^{n-1}}.
 \label{eq:rauch}
\end{equation}
The Jacobi-field estimate applies to every initial angular vector and
therefore gives an inequality on the full tangent space at $x$.
Only curvature along $[p,x]$ enters; nearby geodesics used to
represent the field need not remain in $C$. Here $g_{\SSph^{n-1}}$ is
the round unit-sphere metric. The function
$t\mapsto t\coth t$ is increasing, since
\[
 \frac{\dd}{\dd t}(t\coth t)
 =\frac{\sinh t\cosh t-t}{\sinh^2t}>0.
\]
It follows that
$r\mapsto\sinh(\sqrt\lambda r)/(\sqrt\lambda\sinh r)$ is
nonincreasing. Thus on $C$,
\[
 \alpha_\lambda(R)^2h\leq g\leq h.
\]
The exponential identification is a smooth diffeomorphism on a
neighbourhood of the compact set $C$, so it preserves finite
perimeter. Since $C$ is closed, $\partial^*F\subset C$ up to a
null set. Comparing volume determinants and tangential determinants
on the reduced boundary of the same coordinate set gives
\[
 \Per_g(F)\geq\alpha_\lambda(R)^{n-1}\Per_h(F),\qquad
 \Vol_g(F)\leq\Vol_h(F).
\]
The hyperbolic isoperimetric inequality and
Lemma~\ref{lem:margin} now give
\begin{align*}
 \cQ_g(F)
 &\geq\alpha_\lambda(R)^{n(n-1)}
   \frac{\Per_h(F)^n}{n^n\omega_n\Vol_h(F)^{n-1}}\\
 &\geq\alpha_\lambda(R)^{n(n-1)}Q_h(\Vol_h(F))\\
 &\geq\alpha_\lambda(R)^{n(n-1)}Q_h(\Vol_g(F)).
\end{align*}
This proves \eqref{eq:chart}. 

Finally, the elementary inequality $x\coth x-1\leq x^2/3$ holds
for $x>0$. Indeed,
\[
 \frac{\dd}{\dd x}\left(\left(1+\frac{x^2}{3}\right)\sinh x
 -x\cosh x\right)
 =\frac{x}{3}(x\cosh x-\sinh x)\geq0,
\]
and the expression being differentiated vanishes at zero.
Differentiating with respect to $t>0$ therefore gives
\[
 0\leq\frac{\partial}{\partial t}\log\alpha_t(R)
 =\frac{\sqrt t R\coth(\sqrt t R)-1}{2t}
 \leq\frac{R^2}{6}.
\]
Integrating from $\lambda$ to one and using $\alpha_1(R)=1$
proves \eqref{eq:alphabound}.
\end{proof}

We now use these comparisons to exclude a tightly pinched convex
hull. Take $d_n$ as in \eqref{eq:diameterconstant}, and set
\begin{equation}
 T_n=\frac{\sqrt2\,n}{n-1}.
 \label{eq:volumeradiusscale}
\end{equation}
The number
\begin{equation}
 \varepsilon_n=
 \min\left\{\frac12,
 \frac{3(n-1)^2}{(n+2)\cosh T_n(1+\cosh T_n)d_n^2}\right\}
 \label{eq:epsilon}
\end{equation}
belongs to $(0,1/2]$. The following proposition combines the scale
estimates of Section~\ref{sec:minimiser} with
Lemmata~\ref{lem:margin} and~\ref{lem:chart}. It is the key local
statement used in our proof of Theorem~\ref{thm:main}.

\begin{proposition}
\label{prop:localexclusion}
Let $\Omega$, $\theta<1$ and $C$ be as in
Proposition~\ref{prop:minimiser}. For $\varepsilon_n$ chosen in
\eqref{eq:epsilon}, there is no $b>0$ such that
\begin{equation}
 -b^2\leq\sec_g\leq-(1-\varepsilon_n)b^2
 \quad\text{at every point of }C.
 \label{eq:hullpinching}
\end{equation}
\end{proposition}

\begin{proof}
Suppose that \eqref{eq:hullpinching} holds.
Replace $g$ by $b^2g$. Sectional curvatures are divided by $b^2$,
and the quotient \eqref{eq:quotient} is unchanged.
Writing $H$, $V$ and $P$ for the rescaled quantities,
all the minimiser identities and scale estimates retain their form.
We may therefore assume
\[
 -1\leq\sec_g\leq-\lambda\quad\text{on }C,
 \qquad \lambda=1-\varepsilon_n\geq\frac12.
\]
Set
\[
 a=\left(\frac{V}{\omega_n}\right)^{1/n}
   =\frac{(n-1)\theta^{1/n}}{H},\qquad
 R=\frac{d_n\theta}{H}.
\]
By Corollary~\ref{cor:scale}, $\diam C\leq R$ and
\[
 a\leq\frac{n}{(n-1)\sqrt\lambda}\leq T_n.
\]
Also, since $0<\theta<1$,
\[
 \frac{R^2}{a^2}
 =\frac{d_n^2}{(n-1)^2}\theta^{2-2/n}
 \leq\frac{d_n^2}{(n-1)^2}.
\]
Apply Lemma~\ref{lem:chart} to $\Omega\subset C$, and then use
\eqref{eq:alphabound} and \eqref{eq:quantitativemargin}. We obtain
\begin{align*}
 \log\theta
 &\geq-\frac{n(n-1)\varepsilon_nR^2}{6}
      +\frac{n(n-1)a^2}{(n+2)\cosh T_n(1+\cosh T_n)}\\
 &\geq n(n-1)a^2\left(
 \frac{1}{(n+2)\cosh T_n(1+\cosh T_n)}
 -\frac{\varepsilon_nd_n^2}{6(n-1)^2}\right)\\
 &\geq\frac{n(n-1)a^2}{2(n+2)\cosh T_n(1+\cosh T_n)}>0,
\end{align*}
where the last inequality follows from \eqref{eq:epsilon}.
This contradicts $\theta<1$.
\end{proof}

\begin{corollary}
\label{cor:localball}
Let $B$ be an open geodesic ball in a smooth Cartan--Hadamard
manifold. If \eqref{eq:pinching} holds on $\overline B$, with
$\varepsilon_n$ as in \eqref{eq:epsilon}, then \eqref{eq:CH}
holds for every finite-perimeter set $E\subset B$.
\end{corollary}

\begin{proof}
If a set violates the inequality, minimise the quotient in the same
ball $B$. The proof of Proposition~\ref{prop:minimiser} applies.
The resulting convex hull lies in $\overline B$, contradicting
Proposition~\ref{prop:localexclusion}.
\end{proof}

\section{Proof of Theorem~\ref{thm:main}}
\label{sec:proof}

We now convert slow variation of the curvature radius into constant
pinching on the hull of a possible minimiser. The diameter estimate
turns the scalar variation bound into the pinching required by
Proposition~\ref{prop:localexclusion}.

\begin{proof}[Proof of Theorem~\ref{thm:main}]
Choose $\varepsilon_n$ as in \eqref{eq:epsilon} and set
\begin{equation}
 c_n=\frac{n d_n}{(n-1)^2}.
 \label{eq:variationcoefficient}
\end{equation}
Suppose the hypotheses hold but some bounded finite-perimeter set
violates \eqref{eq:CH}. Put
$\tau=\Lip_g(\kappa^{-1/2})$ and choose the minimiser and its hull
$C$ as in Proposition~\ref{prop:minimiser}. Set
\[
 k_-=\min_C\kappa>0,\qquad k_+=\max_C\kappa,\qquad
 D=\diam C.
\]
The upper curvature bound on $C$ is
$\sec_g\leq-(1-\eta)k_-$. Corollary~\ref{cor:scale} gives
\begin{equation}
 H\geq\frac{(n-1)^2}{n}\sqrt{(1-\eta)k_-},\qquad
 D\leq\frac{d_n}{H}.
 \label{eq:variableH}
\end{equation}
The oscillation of $\ell=\kappa^{-1/2}$ on $C$ satisfies
\[
 \frac1{\sqrt{k_-}}-\frac1{\sqrt{k_+}}
 =\max_C\ell-\min_C\ell\leq\tau D.
\]
Multiplying by $\sqrt{k_-}$ and using \eqref{eq:variableH}, we
obtain
\[
 \sqrt{\frac{k_-}{k_+}}
 \geq1-\tau D\sqrt{k_-}
 \geq1-\frac{n d_n\tau}{(n-1)^2\sqrt{1-\eta}}.
\]
By \eqref{eq:jointcriterion} the right-hand side is positive.
Squaring and multiplying by $1-\eta$ therefore yields
\begin{equation}
 (1-\eta)\frac{k_-}{k_+}
 \geq\left(\sqrt{1-\eta}-\frac{n d_n}{(n-1)^2}\tau\right)^2
 \geq1-\varepsilon_n.
 \label{eq:oscillationpinching}
\end{equation}
Consequently
\[
 -k_+\leq\sec_g\leq-(1-\varepsilon_n)k_+
 \quad\text{on }C.
\]
This contradicts Proposition~\ref{prop:localexclusion}, with
$b=\sqrt{k_+}$.
\end{proof}

\end{document}